\documentclass[amssymb, 11pt]{amsart}
\usepackage{latexsym}
\usepackage{young}
\usepackage{hyperref}

\newlength{\standardunitlength}
\newcommand{\bea}{\begin{eqnarray}}
\newcommand{\ena}{\end{eqnarray}}
\newcommand{\beas}{\begin{eqnarray*}}
\newcommand{\enas}{\end{eqnarray*}}

\newcommand{\ignore}[1]{}

\newtheorem{prop}{Proposition}[section]

\newtheorem{lemma}[prop]{Lemma}
\newtheorem{cor}[prop]{Corollary}
\newtheorem{theorem}[prop]{Theorem}

\begin{document}

\title [The number of involutions] {Asymptotics of the number of involutions in finite classical groups}

\author{Jason Fulman}
\address{Department of Mathematics\\
University of Southern California\\
Los Angeles, CA 90089-2532} \email{fulman@usc.edu}

\author{Robert Guralnick}
\address{Department of Mathematics\\
University of Southern California\\
Los Angeles, CA 90089-2532} \email{guralnic@usc.edu}

\author{Dennis Stanton}
\address{School of Mathematics\\
University of Minnesota\\
Minneapolis, MN 55455} \email{stanton@math.umn.edu}


\date{Submitted February 23, 2016; revised February 18, 2017}


\begin{abstract} Answering a question of Geoff Robinson, we compute the large $n$
limiting proportion of $i_{GL}(n,q)/q^{\lfloor n^2/2 \rfloor}$, where $i_{GL}(n,q)$
denotes the number of
involutions in $GL(n,q)$. We give similar results for the finite unitary,
symplectic, and orthogonal groups, in both odd and even characteristic.
At the heart of this work are certain new ``sum=product" identities. Our self-contained
treatment of the enumeration of involutions in even characteristic symplectic and
orthogonal groups may also be of interest.
\end{abstract}

\maketitle

\section{Introduction}

On math overflow, Dima Pasechnik asked about the asymptotic behavior of the number of involutions
in $GL(n,2)$, as $n \rightarrow \infty$. Again on math overflow, Geoff Robinson observed that for $n$ even, the
number of involutions in $GL(n,2)$ is ``unreasonably close'' to the estimate $2^{n^2/2}$. In an email to us in October
2015, Robinson asked whether for $n$ even we could compute the limiting value of $i_{GL}(n,2)/2^{n^2/2}$, where $i_{GL}(n,2)$
is the number of involutions in $GL(n,2)$, and asserted that he could show that this limiting value should be between
$3/4$ and something probably larger than $2$.

Letting $i_{GL}(n,q)$ denote the number of involutions in $GL(n,q)$, we are able to compute the more general limiting
proportion $i_{GL}(n,q)/q^{\lfloor n^2/2 \rfloor}$. The answer depends on whether $n$ is even or odd and on whether $q$
is even or odd. For example if $n$ is even, and if $q$ is even and fixed, it is shown in Section \ref{GL} that

\begin{eqnarray*}
\lim_{n \rightarrow \infty} \frac{i_{GL}(n,q)}{q^{n^2/2}} & = & \frac{1}{2} \left[ \prod_{i \geq 1} (1+\sqrt{q}/q^i) +
 \prod_{i \geq 1} (1-\sqrt{q}/q^i) \right] \\
 & = & \prod_{i\ge 1} \frac{(1+q^{5-8i})(1+q^{3-8i})(1-q^{-8i})}
{(1-q^{-2i})}.
\end{eqnarray*} When $q=2$ this is $1.6793..$, which is indeed between $3/4$ and something probably larger than $2$.
The reason for dividing by $q^{n^2/2}$ is that for $n$ even, $n^2/2$ is the dimension (in the algebraic group)
of the largest conjugacy class of involutions. Similarly when $n$ is odd, one divides by $q^{(n^2-1)/2}$.
Our asymptotic results are cited by Ginzburg and Pasechnik in their recent work on random chain complexes \cite{GP}.

At the heart of these results for $GL(n,q)$ are the following ``sum=product" identities, proved in \cite{FV}.
These identities state that:

\[  \sum_{n \geq 0} u^n q^{{n \choose 2}} \sum_{r=0}^{\lfloor n/2 \rfloor}
\frac{1}{q^{r(2n-3r)} |GL(r,q)| |GL(n-2r,q)|}  = \frac{\prod_{i \geq 1} (1+u/q^i)}{\prod_{i \geq 1} (1-u^2/q^i)}.\]

\[ \sum_{n \geq 0} u^n q^{{n \choose 2}} \sum_{r=0}^{n}
\frac{1}{|GL(r,q)| |GL(n-r,q)|} = \frac{\prod_{i \geq 1} (1+u/q^i)^2}{\prod_{i \geq 1} (1-u^2/q^i)} .\]

The first identity is useful for even characteristic, and the second identity is useful for odd characteristic.
We give new proofs of these identities, and derive analogous identities for the other finite
classical groups, which we believe to be interesting and new. The proof of the even characteristic
symplectic identity is tricky, as are the orthogonal group identities.

The point of these identities is that it allows us to obtain product formulae for the generating functions for proportions of
involutions in finite classical groups. These product formulae are perfectly suited for asymptotic analysis, using
the following elementary lemma of Darboux (given an exposition in \cite{O}).

\begin{lemma} \label{poles} Suppose that $f(u)$ is analytic for $|u|<r, r>0$ and has a finite number of simple poles on $|u|=r$.
Let $\{ w_j \}$ denote the poles, and suppose that $f(u)=\sum_j \frac{g_j(u)}{1-u/w_j}$ with $g_j(u)$ analytic near $w_j$.
Then the coefficient of $u^n$ in the Taylor expansion of $f(u)$ around $0$ is
\[ \sum_j \frac{g_j(w_j)}{w_j^n} + o(1/r^n) .\]
\end{lemma}

For each of the families of groups we consider, we get an infinite product formula for the limiting value
of $i(n,q)/q^{d(n,q)}$ for $q$ fixed and $n \rightarrow \infty$. Here $d(n,q)$ is the dimension of the
variety of involutions in the corresponding algebraic group. It turns out that the limit
of this formula as $q \rightarrow \infty$
is precisely the number of components of maximal dimension in the variety (in every case this is either $1$
or $2$).

This paper is organized as follows. Section \ref{Identities} derives many ``sum=product" identities which
are crucial to our work. The remaining sections study asymptotics of the number of involutions.
Section \ref{GL} treats the general linear groups, and as an easy corollary, the unitary groups.
Section \ref{Sp} treats the symplectic groups, and Section \ref{O} treats the orthogonal groups.
We mention that Sections \ref{Sp} and \ref{O} give self-contained derivations of the number of involutions
in the symplectic and orthogonal groups. While this is easy in odd characteristic, the enumeration of
involutions in even characteristic is subtle (see \cite{AS}, \cite{D1}, \cite{D2}, \cite{LS} for related results).
Our treatment of involutions is used in a recent article of Garibaldi and Guralnick \cite{GG}.

To close the introduction, we make three remarks about future work. First, it would be interesting to find proofs of
our results which do not require generating functions. Second, our methods will work for the subgroups
$SO$ and $\Omega$ of the orthogonal groups, and should work for the spin groups.
Third, enumeration of involutions is connected
to representation theory via Frobenius-Schur indicators. This is developed for $GL$ and $U$ in \cite{FV},
but much remains to be done for other finite classical groups. For some results for $Sp$ and $O$ (in odd
characteristic), see \cite{V} and the references therein.

\section{Identities} \label{Identities}

This section derives some ``sum=product'' identities which will be crucially applied later in the paper.

Throughout this section (and the rest of this paper) we use the following notation:
$$
(A;q)_n=\prod_{k=0}^{n-1} (1-Aq^k),\quad
(A;q)_\infty=\prod_{k=0}^\infty (1-Aq^k) {\text{ if }} -1<q<1.
$$

Theorem \ref{q-binthm} (the $q$-binomial theorem) will be used throughout. See
page 17 of \cite{A} for a proof.

\begin{theorem}
\label{q-binthm}
If $-1<q<1$ and $|x|<1$, then
$$
\sum_{n=0}^\infty \frac{(A;q)_n}{(q;q)_n} x^n=
\frac{(Ax;q)_\infty}{(x;q)_\infty}.
$$
\end{theorem}

The following two well known special cases are used.

\begin{cor}
\label{cor1}
If $-1<q<1$ and $|x|<1$, then
$$
\sum_{n=0}^\infty \frac{x^n}{(q;q)_n}=
\frac{1}{(x;q)_\infty}.
$$
\end{cor}

\begin{proof} Set $A=0$ in Theorem \ref{q-binthm}.
\end{proof}

\begin{cor}
\label{cor2}
If $-1<q<1$, then
$$
\sum_{n=0}^\infty \frac{x^n q^{\binom{n}{2}}}{(q;q)_n}=
(-x;q)_\infty.
$$
\end{cor}

\begin{proof} Replace $x$ by $-x/A$ in Theorem \ref{q-binthm}, and let $A \rightarrow \infty$.
\end{proof}

Lemma \ref{ourmainlemma} is one of our main lemmas.

\begin{lemma}
\label{ourmainlemma}
If $1<|q|,$ and $|ab/q|<1$, then
$$
\begin{aligned}
\sum_{m=0}^\infty q^{\binom{m}{2}}
& \sum_{k=0}^m \frac{(-a)^{m-k} (-b)^k}{(q;q)_k
q^{\binom{k}{2}}(q;q)_{m-k}q^{\binom{m-k}{2}}}\\
=&
\frac{(-a/q;1/q)_\infty(-b/q;1/q)_\infty }{(ab/q;1/q)_\infty}
=H(a,b,q).
\end{aligned}
$$
\end{lemma}
\begin{proof} Let $Q=1/q.$ Replacing $m$ by $m+k,$ the double sum is
\[
D=\sum_{k=0}^\infty \sum_{m=0}^\infty
\frac{a^{m}b^k}{(Q;Q)_k(Q;Q)_m} Q^{\binom{k+1}{2}+\binom{m+1}{2}-mk}.
\]
The $m$-sum is evaluable by Corollary~\ref{cor2} to
\begin{eqnarray*}
(-aQ^{1-k};Q)_\infty & = & (-aQ^{1-k};Q)_k (-aQ;Q)_\infty \\
& = & (-1/a;Q)_k (-aQ;Q)_\infty a^k Q^{-\binom{k}{2}}.
\end{eqnarray*}
So
$$
D= (-aQ;Q)_\infty \sum_{k=0}^\infty \frac{(-1/a;Q)_k}{(Q;Q)_k} (abQ)^k=
 \frac{(-aQ;Q)_\infty(-bQ;Q)_\infty }{(abQ;Q)_\infty}
$$
by Theorem~\ref{q-binthm}.
\end{proof}

There is a companion lemma which is also useful.

\begin{lemma}
\label{ourmainlemma2}
Suppose that $|q|>1,$ that $s$ is a non-negative even integer, and that
$t$ is a non-negative integer. Let $a,b$ be complex and suppose that
$|bq^{-s/2}|<1$. Then
$$
\begin{aligned}
\sum_{n=0}^\infty q^{\binom{n}{2}}&
\sum_{r=0}^{\lfloor (n-t)/s \rfloor} \frac{a^{n-sr} b^r}{q^{r^2}(1/q;1/q)_r
q^{(n-sr-t)^2}(1/q;1/q)_{n-rs-t}q^{r(sn-(1+s^2/2)r)}}\\
=&
a^tq^{\binom{t}{2}}\frac{(-aq^{t-1};1/q)_\infty}{(bq^{-s/2};1/q)_\infty}
=G(a,b,q,s,t).
\end{aligned}
$$
\end{lemma}

\begin{proof} Replacing $n$ by $n+rs+t$ one obtains
$$
G(a,b,q,s,t)= a^tq^{\binom{t}{2}}
\sum_{r=0}^\infty \frac{\left(bq^{-s/2}\right)^r}{(1/q;1/q)_r}
\sum_{n=0}^\infty \frac{q^{-\binom{n}{2}}\left(aq^{t-1}\right)^n}{(1/q;1/q)_n}.
$$
Then apply Corollaries~\ref{cor1} and ~\ref{cor2}.
\end{proof}

Proposition \ref{sieveJTP} will be useful for writing certain sums and
differences of infinite products as a single infinite product.

\begin{prop}
\label{sieveJTP}
Let $-1<R<1$ and
$$
F(X,R)=(R;R)_\infty (-X;R)_\infty (-R/X;R)_\infty.
$$
Then
$$
\begin{aligned}
\frac{1}{2}\left( F(X,R)+F(-X,R)\right)=& F(RX^2,R^4),\\
\frac{1}{2}\left( F(X,R)-F(-X,R)\right)=& XF(R^3X^2,R^4).
\end{aligned}
$$
\end{prop}
\begin{proof} Use the Jacobi triple product identity (\cite{A}, p.21)
$$
F(X,R)=\sum_{n=-\infty}^\infty R^{\binom{n}{2}}X^n
$$
to find the even and odd parts of $F(X,R).$
\end{proof}

Theorem \ref{sumprodGLeven} will be useful in treating even
characteristic general linear and unitary groups.
Recall that in any characteristic, $|GL(j,q)|=q^{{j \choose 2}} (q^j-1) \cdots (q-1)$.
Theorem \ref{sumprodGLeven} appeared as Corollary 3.5 of \cite{FV},
but for completeness we give a different proof here.

\begin{theorem}
\label{sumprodGLeven} For $q>1$ and $u^2<q$,
\[  \sum_{n \geq 0} u^n q^{{n \choose 2}} \sum_{r=0}^{\lfloor n/2 \rfloor}
\frac{1}{q^{r(2n-3r)} |GL(r,q)| |GL(n-2r,q)|}  =
\frac{\prod_{i \geq 1} (1+u/q^i)}{\prod_{i \geq 1} (1-u^2/q^i)}.\]
\end{theorem}

\begin{proof} Because for any $r \geq 0$,
\begin{equation}
\label{glrcard}
|GL(r,q)|= q^{r^2}(1/q;1/q)_r,
\end{equation}
we use Lemma~\ref{ourmainlemma2} to evaluate the double sum,
which is $G(u,u^2,q,2,0).$
\end{proof}

Theorem \ref{sumprodGLodd} will be helpful in treating odd characteristic general linear and unitary groups.
This result was Corollary 3.7 of \cite{FV}, but we give a different (and simpler) proof here.

\begin{theorem} For $q>1$ and $u^2<q$,
\label{sumprodGLodd}
\[ \sum_{n \geq 0} u^n q^{{n \choose 2}} \sum_{r=0}^{n}
\frac{1}{|GL(r,q)| |GL(n-r,q)|} =
\frac{\prod_{i \geq 1} (1+u/q^i)^2}{\prod_{i \geq 1} (1-u^2/q^i)}.\]
\end{theorem}

\begin{proof} Using \eqref{glrcard}, we can apply Lemma~\ref{ourmainlemma}
to evaluate the double sum,
which is $H(u,u,q).$
\end{proof}

Theorem \ref{prodSp} gives a useful ``sum=product'' formula for the symplectic groups
in odd characteristic. This can also be obtained by replacing $q$ by $q^2$ in
Theorem \ref{sumprodGLodd}.

Recall that in any characteristic, for $n \geq 0$,
\begin{equation}
\label{spcard}
|Sp(2n,q)|=q^{n^2} \prod_{i=1}^n (q^{2i}-1)=q^{2n^2+n}(1/q^2;1/q^2)_n.
\end{equation}

\begin{theorem}
\label{prodSp} For $q>1$ and $|u|<q$,
\[ \sum_{n \geq 0} u^n q^{n^2}
\sum_{r=0}^n \frac{1}{|Sp(2r,q)||Sp(2n-2r,q)|} =
\frac{ \prod_{i \geq 1} (1+u/q^{2i})^2 }{\prod_{i \geq 1} (1-u^2/q^{2i}) } .
\]
\end{theorem}

\begin{proof}
We can apply \eqref{spcard} and
Lemma~\ref{ourmainlemma} to evaluate the double sum, which is
$H(u,u,q^2).$
\end{proof}

The following rather tricky identity will be useful for treating even characteristic
symplectic groups.

\begin{theorem} \label{Guralsum} For $q>1$ and $|u|<1$,
$$
\sum_{n\ge 0} u^nq^{n^2} \left(
\sum_{r=0 \atop r \ even}^n \frac{1}{A_r}
+\sum_{r=1 \atop r \ even}^n \frac{1}{B_r}+
\sum_{r=1 \atop r \ odd}^n \frac{1}{C_r}
\right)=
\frac{1}{1-u} \frac{\prod_{i \geq 1} (1+u/q^{2i})}{\prod_{i \geq 1} (1-u^2/q^{2i})}.
$$
where
$$
\begin{aligned}
A_r= &q^{r(r+1)/2+r(2n-2r)}|Sp(r,q)|\ |Sp(2n-2r,q)|,\\
B_r=& q^{r(r+1)/2+r(2n-2r)}q^{r-1}|Sp(r-2,q)|\ |Sp(2n-2r,q)|,\\
C_r=& q^{r(r+1)/2+r(2n-2r)}|Sp(r-1,q)|\ |Sp(2n-2r,q)|.
\end{aligned}
$$
\end{theorem}

\begin{proof}Each of the three double sums may be computed using
Lemma~\ref{ourmainlemma2}.

Replacing $r$ by $2r$, $2r+2$, and $2r+1$ respectively in these three double sums,
we see that the three sums become
$$
\begin{aligned}
G(u,u^2,& q^2,2,0)+q^6 G(u/q^4,u^2q^2,q^2,2,2) + q^2 G(u/q^2,u^2q^2,q^2,2,1)\\
=&
(-u/q^2;1/q^2)_\infty
\left(\frac{1}{(u^2/q^2;1/q^2)_\infty}+\frac{u^2}{(u^2;1/q^2)_\infty}+
\frac{u}{(u^2;1/q^2)_\infty}
\right)\\
=& \frac{1}{1-u} \frac{(-u/q^2;1/q^2)_\infty}{(u^2/q^2;1/q^2)_\infty}
\end{aligned}
$$
\end{proof}

Next we prove three ``sum=product" identities which will be useful for studying odd characteristic orthogonal groups.
Recall that in any characteristic,
\begin{equation}
\label{Ooddcard}
 |O^+(2n+1,q)| = |O^-(2n+1,q)| = 2 q^{n^2}
\prod_{i=1}^n (q^{2i}-1)=2q^{2n^2+n}(1/q^2;1/q^2)_n.
\end{equation}
and that $|O^+(1,q)|=|O^-(1,q)|=2$.
Recall also that in any characteristic, for $n \geq 1$,
\begin{equation}
\label{Oevencard}
|O^{\pm}(2n,q)|=
\frac{2q^{n^2-n}}{q^n\pm 1}\prod_{i=1}^{n} (q^{2i}-1)
=\frac{2q^{2n^2}}{q^n\pm 1}(1/q^2;1/q^2)_n.
\end{equation} It is also helpful to use \eqref{Oevencard} to define
$|O^{\pm}(2n,q)|$ when $n=0$. Thus
\[ |O^+(0,q)|=1 \ , \ 1/|O^-(0,q)|=0. \]

Our first identity is useful in treating odd characteristic even dimensional orthogonal groups of positive type.

\begin{theorem} \label{prodO1} For $q>1$ and $|u|<1/q$,
\begin{eqnarray*}
& &  \sum_{n \geq 0} u^n q^{n^2} \\
& & \cdot \left[ \sum_{r=0}^{2n} \frac{1}{|O^+(r,q)||O^+(2n-r,q)|} + \sum_{r=1}^{2n-1} \frac{1}{|O^-(r,q)||O^-(2n-r,q)|} \right]  \\
& = & \frac{1}{2 (1-uq)} \frac{\prod_{i \geq 1} (1+u/q^{2(i-1)})^2}{\prod_{i \geq 1} (1-u^2/q^{2(i-1)})}
+ \frac{1}{2} \frac{\prod_{i \geq 1} (1+u/q^{2i-1})^2}{\prod_{i \geq 1} (1-u^2/q^{2(i-1)})}.
\end{eqnarray*}
\end{theorem}

\begin{proof} It suffices to show that the left-hand side of the statement of Theorem \ref{prodO1} is equal to
\[ \frac{1}{2}\left[
\frac{(-u/q;1/q^2)_\infty^2}{(u^2;1/q^2)_\infty} +(1+uq)
\frac{(-u;1/q^2)_\infty^2 }{(q^2u^2;1/q^2)_\infty}
\right] .\]

Because the value of $|O^{+}(r,q)|$ depends on the parity of $r$, we split each sum into two sums, depending on the parity of $r$.
\vskip5pt
CASE 1: $r=2R$, first sum:
$$
S_1=\frac{1}{4}
\sum_{n=0}^\infty (-u)^n q^{n^2} \sum_{R=0}^n \frac{(1+q^R)(1+q^{n-R})}
{q^{R^2-R}(q^2;q^2)_{R} \ q^{(n-R)^2-(n-R)} (q^2;q^2)_{n-R}}
$$
To write this in the form of Lemma~\ref{ourmainlemma}
with $q$ replaced by $q^2$, note that
$$
q^{R^2-R}=q^{2\binom{R}{2}},\quad q^{(n-R)^2-(n-R)}=q^{2\binom{n-R}{2}}
$$
$$
q^{n^2}(1+q^R)(1+q^{n-R})=q^{2\binom{n}{2}}q^n(1+q^R)(1+q^{n-R})
$$
The numerator has 4 terms $q^n, q^{n+R}, q^{2n-R}, q^{2n}.$
We pick $(a,b)$ in Lemma~\ref{ourmainlemma},
so that these four terms appear. The choices are
$(a,b)=(uq,uq)$, $(a,b)=(uq,uq^2),$ $(a,b)=(uq^2,uq),$
and $(a,b)=(uq^2,uq^2)$ respectively to obtain
$$
S_1=\frac{1}{4}\left[
\frac{(-u/q;1/q^2)_\infty^2}{(u^2;1/q^2)_\infty}+2
\frac{(-u/q;1/q^2)_\infty (-u;1/q^2)_\infty}{(qu^2;1/q^2)_\infty}+
\frac{(-u;1/q^2)_\infty^2 }{(q^2u^2;1/q^2)_\infty}\right]
$$
\vskip5pt
CASE 2: $r=2R+1$, first sum:
$$
S_2=
\frac{1}{4}
\sum_{n=1}^\infty u^n q^{n^2} \sum_{R=0}^{n-1} \frac{(-1)^R (-1)^{n-1-R}}
{q^{R^2}(q^2;q^2)_R\  q^{(n-R-1)^2} (q^2;q^2)_{n-R-1}}
$$
We use Lemma~\ref{ourmainlemma} with $q$ replaced by $q^2$ and $(a,b)=(uq^2,uq^2)$,
$$
S_2=
\frac{uq}{4}\frac{(-u;1/q^2)_\infty^2 }{(u^2q^2;1/q^2)_\infty}.
$$
\vskip5pt
CASE 3: $r=2R$, second sum:
$$
S_1=\frac{1}{4}
\sum_{n=0}^\infty (-u)^n q^{n^2} \sum_{R=0}^{n} \frac{(q^R-1)(q^{n-R}-1)}
{q^{R^2-R}(q^2;q^2)_R\  q^{(n-R)^2-(n-R)} (q^2;q^2)_{n-R}}
$$
The numerator has 4 terms $q^n, -q^{n+R}, -q^{2n-R}, q^{2n}.$ As in CASE 1 we obtain
$$
S_3=\frac{1}{4}\left[
\frac{(-u/q;1/q^2)_\infty^2}{(u^2;1/q^2)_\infty}-2
\frac{(-u/q;1/q^2)_\infty (-u;1/q^2)_\infty}{(qu^2;1/q^2)_\infty}+
\frac{(-u;1/q^2)_\infty^2 }{(q^2u^2;1/q^2)_\infty}\right]
$$
\vskip5pt
CASE 4: $r=2R+1$, second sum:
$$
S_4=
\frac{1}{4}
\sum_{n=1}^\infty u^n q^{n^2} \sum_{R=0}^{n-1} \frac{(-1)^R (-1)^{n-1-R}}
{q^{R^2}(q^2;q^2)_R\  q^{(n-R-1)^2} (q^2;q^2)_{n-R-1}}
$$
This is the same as CASE 2,
$$
S_4=\frac{uq}{4}\frac{(-u;1/q^2)_\infty^2 }{(u^2q^2;1/q^2)_\infty}.
$$

So
$$
S_1+S_2+S_3+S_4=\frac{1}{2}\left[
\frac{(-u/q;1/q^2)_\infty^2}{(u^2;1/q^2)_\infty} +(1+uq)
\frac{(-u;1/q^2)_\infty^2 }{(q^2u^2;1/q^2)_\infty}
\right]
$$
\end{proof}

Our second identity is useful in treating odd characteristic even dimensional orthogonal groups of negative type.

\begin{theorem} \label{prodO2} For $q>1$ and $|u|<1/q$,
\begin{eqnarray*}
& & \sum_{n \geq 0} u^n q^{n^2} \\
& & \cdot \left[ \sum_{r=0}^{2n-1} \frac{1}{|O^+(r,q)||O^-(2n-r,q)|} + \sum_{r=1}^{2n} \frac{1}{|O^-(r,q)||O^+(2n-r,q)|} \right] \\
& = & \frac{1}{2 (1-uq)} \frac{\prod_{i \geq 1} (1+u/q^{2(i-1)})^2}{\prod_{i \geq 1} (1-u^2/q^{2(i-1)})}
- \frac{1}{2} \frac{\prod_{i \geq 1} (1+u/q^{2i-1})^2}{\prod_{i \geq 1} (1-u^2/q^{2(i-1)})}.
\end{eqnarray*}
\end{theorem}

\begin{proof} It suffices to prove that the left hand side of Theorem \ref{prodO2} is equal to
\[ \frac{1}{2}\left[
-\frac{(-u/q;1/q^2)_\infty^2}
{(u^2;1/q^2)_\infty} +(1+uq)\frac{(-u;1/q^2)_\infty^2 }{(u^2q^2;1/q^2)_\infty}
\right]. \]

Note that the two sums on the left hand side of the theorem are equal, so we focus on the first sum.
Because the value of $|O^{+}(r,q)|$ depends on the parity of $r$, we split the first sum into two cases,
depending on the parity of $r$.
\vskip5pt
CASE 1: $r=2R$:
$$
S_1=\frac{1}{4}
\sum_{n=0}^\infty (-u)^n q^{n^2} \sum_{R=0}^n \frac{(1+q^R)(q^{n-R}-1)}
{q^{R^2-R}(q^2;q^2)_{R} \ q^{(n-R)^2-(n-R)} (q^2;q^2)_{n-R}}
$$
The numerator has 4 terms $-q^n, -q^{n+R}, q^{2n-R}, q^{2n}.$
The middle two terms cancel,
so as in CASE 1 of
Theorem \ref{prodO1}
$$
S_1=\frac{1}{4}\left[
-\frac{(-u/q;1/q^2)_\infty^2}{(u^2;1/q^2)_\infty}+
\frac{(-u;1/q^2)_\infty^2 }{(q^2u^2;1/q^2)_\infty}\right]
$$
\vskip5pt
CASE 2: $r=2R+1$:
As in CASE 2 of Theorem \ref{prodO1}

\begin{eqnarray*}
S_2 & = &
\frac{1}{4}
\sum_{n=1}^\infty u^n q^{n^2} \sum_{R=0}^{n-1} \frac{(-1)^R (-1)^{n-1-R}}
{q^{R^2}(q^2;q^2)_R\  q^{(n-R-1)^2} (q^2;q^2)_{n-R-1}} \\
& = &
\frac{uq}{4}\frac{(-u;1/q^2)_\infty^2 }{(u^2q^2;1/q^2)_\infty}
\end{eqnarray*}

So we have
$$
2(S_1+S_2) = \frac{1}{2}\left[
-\frac{(-u/q;1/q^2)_\infty^2}
{(u^2;1/q^2)_\infty} +(1+uq)\frac{(-u;1/q^2)_\infty^2 }{(u^2q^2;1/q^2)_\infty}
\right]
$$
\end{proof}

The following ``sum=product" formula will be helpful in treating odd dimensional
orthogonal groups in odd characteristic.

\begin{theorem} \label{identlast} For $q>1$ and $|u|<1$,
\begin{eqnarray*}
& & \sum_{n \geq 0} u^n q^{n^2} \cdot \\
& & \left[ \sum_{r=0}^{2n+1} \frac{1}{|O^+(r,q)||O^+(2n+1-r,q)|}
+ \sum_{r=1}^{2n} \frac{1}{|O^-(r,q)||O^-(2n+1-r,q)|} \right]
\end{eqnarray*} is equal to
\[ \frac{1}{1-u} \frac{\prod_{i \geq 1} (1+u/q^{2i})^2}{\prod_{i \geq 1} (1-u^2/q^{2i})} .\]

\end{theorem}

\begin{proof} It suffices to prove that the left hand side is equal to
\[ \frac{(-u;1/q^2)_\infty (-u/q^2;1/q^2)_\infty}{(u^2;1/q^2)_\infty}. \]

We split each sum into two sums, depending on the parity of $r$.
\vskip5pt
CASE 1: $r=2R$, first sum:
$$
S_1=\frac{1}{4}
\sum_{n=0}^\infty (-u)^n q^{n^2} \sum_{R=0}^n \frac{q^{-(n-R)}(1+q^R)}
{q^{R^2-R}(q^2;q^2)_{R} \ q^{(n-R)^2-(n-R)} (q^2;q^2)_{n-R}}
$$
This is nearly CASE 1 of Theorem \ref{prodO1},
the numerator has 2 terms $q^{R-n}$ and $q^{2R-n}.$ We apply
Lemma~\ref{ourmainlemma} with $(a,b)=(u,uq),$ $(a,b)=(u,uq^2),$
and $q$ replaced by $q^2$ to obtain
$$
S_1=\frac{1}{4}\left[
\frac{(-u/q^2;1/q^2)_\infty (-u/q;1/q^2)_\infty}{(u^2/q;1/q^2)_\infty}+
\frac{(-u/q^2;1/q^2)_\infty (-u;1/q^2)_\infty}{(u^2;1/q^2)_\infty}
\right]
$$
\vskip5pt
CASE 2: $r=2R+1$, first sum:
$$
\begin{aligned}
S_2=&
\frac{1}{4}
\sum_{n=0}^\infty (-u)^n q^{n^2} \sum_{R=0}^{n} \frac{(1+q^{n-R})}
{q^{R^2}(q^2;q^2)_R\  q^{(n-R)^2-(n-R)} (q^2;q^2)_{n-R}} \\
=&
\frac{1}{4}\left[
\frac{(-u;1/q^2)_\infty (-u/q^2;1/q^2)_\infty}{(u^2;1/q^2)_\infty}+
\frac{(-u/q;1/q^2)_\infty (-u/q^2;1/q^2)_\infty}{(u^2/q;1/q^2)_\infty}
\right].
\end{aligned}
$$
where we used Lemma \ref{ourmainlemma} with
$(a,b)=(uq^2,u),$ $(a,b)=(uq,u)$ and $q$ replaced by $q^2$.
\vskip5pt
CASE 3: $r=2R$, second sum:
$$
S_3=\frac{1}{4}
\sum_{n=0}^\infty (-u)^n q^{n^2} \sum_{R=0}^{n} \frac{(q^R-1)}
{q^{R^2-R}(q^2;q^2)_R\  q^{(n-R)^2} (q^2;q^2)_{n-R}}
$$
The numerator has 2 terms $q^nq^{-(n-R)}(q^R-1)=q^{2R}-q^R.$
Using $(a,b)=(u,uq^2),$ $(a,b)=(u,uq)$ and $q$ replaced by $q^2$
in Lemma \ref{ourmainlemma}, we obtain
$$
S_3=\frac{1}{4}\left[
\frac{(-u/q^2;1/q^2)_\infty (-u;1/q^2)_\infty}{(u^2;1/q^2)_\infty}-
\frac{(-u/q^2;1/q^2)_\infty (-u/q;1/q^2)_\infty}{(u^2/q;1/q^2)_\infty}
\right]
$$
\vskip5pt
CASE 4: $r=2R+1$, second sum:
$$
S_4=
\frac{1}{4}
\sum_{n=0}^\infty (-u)^n q^{n^2} \sum_{R=0}^{n} \frac{(q^{n-R}-1)}
{q^{R^2}(q^2;q^2)_R\  q^{(n-R)^2-(n-R)} (q^2;q^2)_{n-R}}
$$
This is basically the same as CASE 2,
$$
S_4=\frac{1}{4}\left[
\frac{(-u;1/q^2)_\infty (-u/q^2;1/q^2)_\infty}{(u^2;1/q^2)_\infty}-
\frac{(-u/q;1/q^2)_\infty (-u/q^2;1/q^2)_\infty}{(u^2/q;1/q^2)_\infty}
\right].
$$

So we have
$$
S_1+S_2+S_3+S_4=
\frac{(-u;1/q^2)_\infty (-u/q^2;1/q^2)_\infty}{(u^2;1/q^2)_\infty}.
$$
\end{proof}

Finally, we give two identities which will be useful for analyzing the number of
involutions in even characteristic orthogonal groups.

Theorem \ref{GuralsumO+} will be useful for positive type orthogonal groups in even characteristic.

\begin{theorem} \label{GuralsumO+} For $q>1$ and $|u|<1/q$,
\begin{eqnarray*}
& & \sum_{n \geq 0} q^{n^2}u^n \left[ \sum_{r=0 \atop r \ even}^n \frac{1}{A_r} +
\sum_{r=1 \atop r \ even}^n \frac{1}{B_r} + \sum_{r=1 \atop r \ odd}^n \frac{1}{C_r}
    \right] \\
& = & \frac{1}{2(1-uq)}
\frac{\prod_{i \geq 1} (1+u/q^{2(i-1)})}{\prod_{i \geq 1} (1-u^2/q^{2(i-1)})}
+ \frac{1}{2}
\frac{\prod_{i \geq 1} (1+u/q^{2i-1})}{\prod_{i \geq 1} (1-u^2/q^{2(i-1)})},
\end{eqnarray*} where

\[ A_r = q^{r(r-1)/2+r(2n-2r)} |Sp(r,q)| |O^+(2n-2r,q)| \]

\[ B_r = 2q^{r(r+1)/2+(r-1)(2n-2r)-1} |Sp(r-2,q)| |Sp(2n-2r,q)| \]

\[ C_r = 2q^{r(r-1)/2+(r-1)(2n-2r)} |Sp(r-1,q)| |Sp(2n-2r,q)|. \]
\end{theorem}

\begin{proof} The first double sum, after replacing $r$ by $2r$, may be rewritten,
after using \eqref{spcard}, \eqref{Oevencard}, and Lemma~\ref{ourmainlemma2}, as
$$
\frac{1}{2}\left( G(uq^2,u^2q^2,q^2,2,0)+G(uq,u^2q^2,q^2,2,0)\right).
$$
The last two double sums, after replacing $r$ by $2r+2$ and $2r+1$
respectively, are
$$
\frac{1}{2}\left( q^4G(u/q^2,u^2q^4,q^2,2,2)+q^1 G(u,u^2q^4,q^2,2,1)\right).
$$
The identity
$$
\begin{aligned}
\frac{1}{2}&\left(G(uq^2,u^2q^2,q^2,2,0)+ q^4G(u/q^2,u^2q^4,q^2,2,2)
+q^1 G(u,u^2q^4,q^2,2,1)\right)\\
=&
\frac{1}{2(1-uq)} \frac{(-u;1/q^2)_\infty}{(u^2;1/q^2)_\infty}
\end{aligned}
$$
completes the proof.
\end{proof}

Theorem \ref{GuralsumO-} will be useful for negative type orthogonal groups in even characteristic.

\begin{theorem}
\label{GuralsumO-} For $q>1$ and $|u|<1/q$,
\begin{eqnarray*}
& & \sum_{n \geq 0} q^{n^2}u^n \left[ \sum_{r=0 \atop r \ even}^{n-1} \frac{1}{A_r} + \sum_{r=1 \atop r \ even}^n \frac{1}{B_r} + \sum_{r=1 \atop r \ odd}^n \frac{1}{C_r}
    \right] \\
& = & \frac{1}{2(1-uq)} \frac{\prod_{i \geq 1} (1+u/q^{2(i-1)})}{\prod_{i \geq 1} (1-u^2/q^{2(i-1)})}
- \frac{1}{2} \frac{\prod_{i \geq 1} (1+u/q^{2i-1})}{\prod_{i \geq 1} (1-u^2/q^{2(i-1)})},
\end{eqnarray*} where

\[ A_r = q^{r(r-1)/2+r(2n-2r)} |Sp(r,q)| |O^-(2n-2r,q)| \]

\[ B_r = 2q^{r(r+1)/2+(r-1)(2n-2r)-1} |Sp(r-2,q)| |Sp(2n-2r,q)| \]

\[ C_r = 2q^{r(r-1)/2+(r-1)(2n-2r)} |Sp(r-1,q)| |Sp(2n-2r,q)|. \]
\end{theorem}

\begin{proof} The only change to the proof of Theorem~\ref{GuralsumO+},
is that the first double sum has a minus sign in the numerator factor for
$|O^-(2n-2r,q)|$. The sum for $1/A_r$ can be extended to $r=n$ since
$1/|O^-(0,q)|=0$ by \eqref{Oevencard}.
\end{proof}

\section{General linear and unitary groups} \label{GL}

To begin we study asymptotics of the number of involutions in $GL(n,q)$, when $q$ is fixed and $n \rightarrow \infty$.
Throughout we let $i_{GL}(n,q)$ denote the number of involutions in $GL(n,q)$. The following explicit formula for $i_{GL}(n,q)$
can be found in Section 1.11 of \cite{M}.

\begin{prop} \label{GLinv} The number of involutions in $GL(n,q)$ is
$$
i_{GL}(n,q)=
\begin{cases}
\sum_{r=0}^{\lfloor n/2 \rfloor}
\frac{|GL(n,q)|}{q^{r(2n-3r)} |GL(r,q)| |GL(n-2r,q)|}, {\text{ for $q$ even,}}\\
\sum_{r=0}^{n} \frac{|GL(n,q)|}
{|GL(r,q)| |GL(n-r,q)|},{\text{ for $q$ odd.}}
\end{cases}
$$
\end{prop}

To begin we treat the case of even characteristic.

\begin{theorem} \label{GLlimeven}
\begin{enumerate}
\item Let $n$ be even and let $q$ be even and fixed. Then
\begin{eqnarray*}
\lim_{n \rightarrow \infty} \frac{i_{GL}(n,q)}{q^{n^2/2}} & = &
\frac{1}{2} \left[ \prod_{i \geq 1} (1+\sqrt{q}/q^i) +
 \prod_{i \geq 1} (1-\sqrt{q}/q^i) \right] \\
 & = & \prod_{i\ge 1} \frac{(1+q^{5-8i})(1+q^{3-8i})(1-q^{-8i})}
{(1-q^{-2i})}.
\end{eqnarray*}

\item Let $n$ be odd and let $q$ be even and fixed. Then
\begin{eqnarray*}
\lim_{n \rightarrow \infty} \frac{i_{GL}(n,q)}{q^{(n^2-1)/2}} & = &
\frac{\sqrt{q}}{2} \left[ \prod_{i \geq 1} (1+\sqrt{q}/q^i) -
 \prod_{i \geq 1} (1-\sqrt{q}/q^i) \right] \\
& = & \prod_{i\ge 1} \frac{(1+q^{7-8i})(1+q^{1-8i})(1-q^{-8i})}
{(1-q^{-2i})}.
\end{eqnarray*}
\end{enumerate}
\end{theorem}

\begin{proof} Replacing $u$ by $u \sqrt{q}$ in Theorem \ref{sumprodGLeven} gives that
\begin{equation} \label{a}
\sum_{n \geq 0} u^n q^{n/2} q^{{n \choose 2}} \sum_{r=0}^{\lfloor n/2 \rfloor}
\frac{1}{q^{r(2n-3r)} |GL(r,q)| |GL(n-2r,q)|}
\end{equation} is equal to
\begin{equation} \label{b}
\frac{\prod_i (1+u \sqrt{q}/q^i)}{(1-u^2) \prod_i (1-u^2/q^i)}.
\end{equation}

First consider the $n \rightarrow \infty$ limit of the coefficient of $u^n$ in \eqref{a}. This is equal to
\begin{eqnarray*}
& & \lim_{n \rightarrow \infty} \frac{q^{n/2}}{(q^n-1) \cdots (q-1)}
\sum_{r=0}^{\lfloor n/2 \rfloor} \frac{|GL(n,q)|}
{q^{r(2n-3r)} |GL(r,q)| |GL(n-2r,q)|} \\
& = & \lim_{n \rightarrow \infty} \frac{q^{n/2}}{(q^n-1) \cdots (q-1)} i_{GL}(n,q) \\
& = & \frac{1}{\prod_i (1-1/q^i)} \lim_{n \rightarrow \infty} \frac{i_{GL}(n,q)}{q^{n^2/2}}.
\end{eqnarray*}  The first equality is from Proposition \ref{GLinv}.

Now consider \eqref{b}. Except for poles at $u=\pm 1$, it is
analytic in a disc of radius greater than 1.
Hence Lemma \ref{poles} gives that for $n$ even, the
$n \rightarrow \infty$ limit of the coefficient of $u^n$ in
\eqref{b} is
\begin{equation}
\frac{1}{2} \left[ \frac{\prod_i (1+\sqrt{q}/q^i)}{\prod_i (1-1/q^i)} +
 \frac{\prod_i (1-\sqrt{q}/q^i)}{\prod_i (1-1/q^i)} \right].
\end{equation}
It follows that for $n$ even,
\[
\lim_{n \rightarrow \infty}
\frac{i_{GL}(n,q)}{q^{n^2/2}} = \frac{1}{2}
\left[ \prod_{i \geq 1} (1+\sqrt{q}/q^i) +
 \prod_{i \geq 1} (1-\sqrt{q}/q^i) \right]. \] Letting $Q=1/q$, this is equal to

\begin{eqnarray*}
& & \frac{1}{2}
\left( (-Q^{1/2};Q)_\infty +(Q^{1/2};Q)_\infty\right)\\
& = & \frac{1}{2}
\left( (-Q^{1/2};Q^2)_\infty(-Q^{3/2};Q^2)_\infty
+(Q^{1/2};Q^2)_\infty(Q^{3/2};Q^2)_\infty\right)\\
& = &\frac{1}{2}\frac{1}{(Q^2;Q^2)_\infty}
\left( F(Q^{1/2},Q^2) +F(-Q^{1/2},Q^2)\right)\\
& = & \frac{1}{(Q^2;Q^2)_\infty} F(Q^3,Q^8),
\end{eqnarray*} where $F$ is defined in Proposition \ref{sieveJTP} and where in the last step we have applied Proposition~\ref{sieveJTP}
with $X=Q^{1/2}$ and $R=Q^2.$

Similarly, if $n$ is odd, Lemma \ref{poles} gives that the $n \rightarrow \infty$ limit of the coefficient of $u^n$ in
\eqref{b} is
\begin{equation}  \frac{1}{2}
\left[
\frac{\prod_i (1+\sqrt{q}/q^i)}{\prod_i (1-1/q^i)} -
 \frac{\prod_i (1-\sqrt{q}/q^i)}{\prod_i (1-1/q^i)}.
 \right]
\end{equation} It follows that for $n$ odd,
\[ \lim_{n \rightarrow \infty} \frac{i_{GL}(n,q)}{q^{n^2/2}} =
\frac{1}{2} \left[ \prod_{i \geq 1} (1+\sqrt{q}/q^i) -
 \prod_{i \geq 1} (1-\sqrt{q}/q^i) \right] .\] Letting $Q=1/q$ and arguing as in the $n$ even
 case gives that this is equal to
 \[ \frac{1}{2} \frac{1}{(Q^2;Q^2)_\infty}
\left( F(Q^{1/2},Q^2) -F(-Q^{1/2},Q^2)\right) = \frac{1}{\sqrt{q}} \frac{1}{(Q^2;Q^2)_\infty} F(Q^7,Q^8).
 \] This proves the theorem. \end{proof}

{\it Remarks:}

\begin{enumerate}

\item The expression in part 1 of Theorem \ref{GLlimeven} is equal to $1.6793..$ when $q=2$, and
tends to $1$ as $q \rightarrow \infty$. The expression in part 2 of Theorem \ref{GLlimeven} is equal
to $2.1912..$ when $q=2$, and tends to $1$ as $q \rightarrow \infty$.

\item In parts 1 and 2 of Theorem \ref{GLlimeven}, one can rewrite the $(1-q^{-8i})/(1-q^{-2i})$
as $(1+q^{-4i})(1+q^{-2i})$. This makes it clear that the limits in the theorem decrease
monotonically to $1$ as $q \rightarrow \infty$. Our reason for not doing this in the statement
of Theorem \ref{GLlimeven} is to emphasize the role in the numerator played by the base $8$.

\end{enumerate}

Next we treat odd characteristic. Our main result is the following theorem.

\begin{theorem} \label{GLlimodd}
\begin{enumerate}
\item Let $n$ be even and let $q$ be odd and fixed. Then
\begin{eqnarray*}
\lim_{n \rightarrow \infty}
\frac{i_{GL}(n,q)}{q^{n^2/2}} & = & \frac{1}{2} \left[ \prod_{i \geq 1} (1+\sqrt{q}/q^i)^2 +
 \prod_{i \geq 1} (1-\sqrt{q}/q^i)^2 \right] \\
& = & \prod_{i\ge 1} \frac{(1+q^{2-4i})^2(1-q^{-4i})}{(1-q^{-i})}.
\end{eqnarray*}

\item Let $n$ be odd and let $q$ be odd and fixed. Then
\begin{eqnarray*}
\lim_{n \rightarrow \infty}
\frac{i_{GL}(n,q)}{q^{(n^2-1)/2}} & = & \frac{\sqrt{q}}{2} \left[ \prod_{i \geq 1} (1+\sqrt{q}/q^i)^2 -
 \prod_{i \geq 1} (1-\sqrt{q}/q^i)^2 \right] \\
& = & 2 \prod_{i\ge 1} \frac{(1+q^{-4i})(1-q^{-8i})}{(1-q^{-i})}.
\end{eqnarray*}
\end{enumerate}
\end{theorem}

\begin{proof} Replacing $u$ by $u \sqrt{q}$ in Theorem \ref{sumprodGLodd} gives the equation

\begin{equation} \label{big2}
\sum_{n \geq 0} u^n q^{n/2} q^{{n \choose 2}}
\sum_{r=0}^{n}\frac{1}{|GL(r,q)| |GL(n-r,q)|} =
\frac{\prod_i (1+u \sqrt{q}/q^i)^2}{(1-u^2) \prod_i (1-u^2/q^i)}.
\end{equation}

Consider the $n \rightarrow \infty$ limit of the coefficient of $u^n$
in the left hand side of \eqref{big2}. This is equal to
\begin{eqnarray*}
& & \lim_{n \rightarrow \infty}
\frac{q^{n/2}}{(q^n-1) \cdots (q-1)} \sum_{r=0}^{n} \frac{|GL(n,q)|}
{|GL(r,q)| |GL(n-r,q)|} \\
& = & \lim_{n \rightarrow \infty} \frac{q^{n/2}}{(q^n-1) \cdots (q-1)} i_{GL}(n,q) \\
& = & \frac{1}{\prod_i (1-1/q^i)} \lim_{n \rightarrow \infty} \frac{i_{GL}(n,q)}{q^{n^2/2}}.
\end{eqnarray*}  The first equality is from Proposition \ref{GLinv}.

Now consider the right hand side of \eqref{big2}.
Except for poles at $u=\pm 1$, it is analytic in a disc of radius greater than 1.
Hence Lemma \ref{poles} gives that for $n$ even, the $n \rightarrow \infty$ limit of the coefficient of $u^n$ in the
right hand side of \eqref{big2} is
\[  \frac{1}{2} \left[ \frac{\prod_i (1+\sqrt{q}/q^i)^2}{\prod_i (1-1/q^i)} +
 \frac{\prod_i (1-\sqrt{q}/q^i)^2}{\prod_i (1-1/q^i)} \right].
\] It follows that for $n$ even,
\[ \lim_{n \rightarrow \infty} \frac{i_{GL}(n,q)}{q^{n^2/2}} =
\frac{1}{2} \left[ \prod_{i \geq 1} (1+\sqrt{q}/q^i)^2 +
 \prod_{i \geq 1} (1-\sqrt{q}/q^i)^2 \right].\] Letting $Q=1/q$ and $F$
 be as in Proposition \ref{sieveJTP}, this is equal to
 \[  \frac{1}{2} \frac{1}{(Q;Q)_\infty}\left( F(Q^{1/2},Q)+F(-Q^{1/2},Q)\right)
 = \frac{1}{(Q;Q)_\infty} F(Q^2,Q^4), \] where we have used Proposition~\ref{sieveJTP}.

Similarly, if $n$ is odd, Lemma \ref{poles} gives that the $n \rightarrow \infty$ limit of the coefficient of $u^n$ in the
right hand side of \eqref{big2} is
\[
\frac{1}{2} \left[
\frac{\prod_i (1+\sqrt{q}/q^i)^2}{\prod_i (1-1/q^i)} -
\frac{\prod_i (1-\sqrt{q}/q^i)^2}{\prod_i (1-1/q^i)}
 \right]. \]
It follows that for $n$ odd,
\[ \lim_{n \rightarrow \infty} \frac{i_{GL}(n,q)}{q^{n^2/2}} =
\frac{1}{2} \left[ \prod_{i \geq 1} (1+\sqrt{q}/q^i)^2 -
 \prod_{i \geq 1} (1-\sqrt{q}/q^i)^2 \right] .\] Letting $Q=1/q$, this is equal to
\[ \frac{1}{2} \frac{1}{(Q;Q)_{\infty}} \left( F(Q^{1/2},Q)-F(-Q^{1/2},Q)\right)=
\frac{1}{\sqrt{q}} \frac{1}{(Q;Q)_\infty} F(Q^4,Q^4).
 \]

This proves the theorem.
\end{proof}

{\it Remark:} The expression in part 1 of Theorem \ref{GLlimodd} is equal to $2.1825..$ when
$q=3$, and tends to $1$ as $q \rightarrow \infty$. The expression in part 2 of Theorem
\ref{GLlimodd} is equal to $3.6147..$ when $q=3$, and tends to $2$ as $q \rightarrow \infty$.

Next we study asymptotics of the number of involutions in the finite unitary groups $U(n,q)$.
Recall that the (general) unitary group is defined to be the subgroup of $GL(n,q^2)$
stabilizing a non degenerate Hermitian form (any two such forms are equivalent and so
the group is unique up to conjugacy -- typically we assume the standard basis is an
orthonormal basis with respect to the Hermitian form).
The results will be an easy consequence of the general linear case.

Throughout we let $i_U(n,q)$ denote the number of involutions in $U(n,q)$.
Recall that in any characteristic
\begin{equation}
\label{unitcard}
|U(n,q)| = q^{n(n-1)/2} \prod_{i=1}^n (q^i-(-1)^i) = (-1)^n |GL(n,-q)|
\end{equation}

The following elementary fact was Proposition 4.4 of \cite{FV} in even characteristic
and was Proposition 4.9 of \cite{FV} in odd characteristic.

\begin{prop} \label{Uinv} The number of involutions in $U(n,q)$ is
$$
i_U(n,q)=
\begin{cases}
\sum_{r=0}^{\lfloor n/2 \rfloor}
\frac{|U(n,q)|}{q^{r(2n-3r)} |U(r,q)| |U(n-2r,q)|}, {\text{ for $q$ even,}}\\
\sum_{r=0}^{n} \frac{|U(n,q)|}
{|U(r,q)| |U(n-r,q)|},{\text{ for $q$ odd.}}
\end{cases}
$$
\end{prop}

Theorem \ref{Uqeven} treats the asymptotics of involutions in $U(n,q)$ with $q$ even.

\begin{theorem} \label{Uqeven}
\begin{enumerate}
\item Let $n$ be even and let $q$ be even and fixed. Then
\[ \lim_{n \rightarrow \infty} \frac{i_U(n,q)}{q^{n^2/2}} = \prod_{k \geq 1}
\frac{(1-q^{3-8k})(1-q^{5-8k})(1-q^{-8k})}{(1-q^{-2k})} .\]

\item Let $n$ be odd and let $q$ be even and fixed. Then
\[ \lim_{n \rightarrow \infty} \frac{i_U(n,q)}{q^{(n^2-1)/2}} = \prod_{k \geq 1}
 \frac{(1-q^{1-8k})(1-q^{7-8k})(1-q^{-8k})}{(1-q^{-2k})}.\]
\end{enumerate}
\end{theorem}

\begin{proof} For part 1, it follows from Propositions \ref{GLinv} and \ref{Uinv} that
for $n$ and $q$ even, $\frac{i_U(n,q)}{q^{n^2/2}}$ is obtained by replacing $q$ by $-q$
in $\frac{i_{GL}(n,q)}{q^{n^2/2}}$. The result now follows from part 1 of Theorem
\ref{GLlimeven}.

For part 2, it follows from Propositions \ref{GLinv} and \ref{Uinv} that
for $n$ odd and $q$ even, $\frac{i_U(n,q)}{q^{(n^2-1)/2}}$ is obtained by replacing $q$ by $-q$
in $\frac{i_{GL}(n,q)}{q^{(n^2-1)/2}}$. The result now follows from part 2 of Theorem
\ref{GLlimeven}. \end{proof}

Next we treat the case of odd characteristic unitary groups.

\begin{theorem} \label{Uqodd}
\begin{enumerate}
\item Let $n$ be even and let $q$ be odd and fixed. Then
\[ \lim_{n \rightarrow \infty} \frac{i_U(n,q)}{q^{n^2/2}} =
\prod_{k \geq 1} \frac {(1+q^{2-4k})^2(1-q^{-4k})}{(1-(-1/q)^{k})}.\]

\item Let $n$ be odd and let $q$ be odd and fixed. Then
\[ \lim_{n \rightarrow \infty} \frac{i_U(n,q)}{q^{(n^2-1)/2}} =
2 \prod_{k \geq 1} \frac {(1+q^{-4k})^2(1-q^{-4k})}{(1-(-1/q)^{k})}.\]
\end{enumerate}
\end{theorem}

\begin{proof} For part 1, it follows from Propositions \ref{GLinv} and \ref{Uinv}
that for $n$ even and $q$ odd, $\frac{i_U(n,q)}{q^{n^2/2}}$ is obtained by replacing $q$ by $-q$
in $\frac{i_{GL}(n,q)}{q^{n^2/2}}$. The result now follows from part 1 of Theorem
\ref{GLlimodd}.

For part 2, it follows from Propositions \ref{GLinv} and \ref{Uinv} that
for $n$ odd and $q$ odd, $\frac{i_U(n,q)}{q^{(n^2-1)/2}}$ is obtained by replacing $q$ by $-q$
in $\frac{i_{GL}(n,q)}{q^{(n^2-1)/2}}$. The result now follows from part 2 of Theorem
\ref{GLlimodd}.
\end{proof}

\section{Symplectic groups} \label{Sp}

This section studies the asymptotics of the number of involutions
in the finite symplectic groups. We begin with the case of odd
characteristic. Lemma \ref{countSp} gives a formula for the number of involutions
(see also the proof of Theorem 3.1 of \cite{V}).

\begin{lemma} \label{countSp} Suppose that $q$ is odd. Then the number of involutions in $Sp(2n,q)$ is equal to
\[  \sum_{r=0}^n \frac{|Sp(2n,q)|}{|Sp(2r,q)| |Sp(2n-2r,q)|} .\]
\end{lemma}

\begin{proof} Let $g \in Sp(2n,q)$  be an involution.   Let $V$ denote the natural $2n$-dimensional
module.   Then $V = V_{-1} \perp V_1$ where $V_a$ is the eigenspace of $g$ with eigenvalue $a$.
Note that each $V_1$ is a nondegenerate space of even dimension.
Set $2r = \dim V_1$.    Since any two nondegenerate spaces of the same dimension are in the same $Sp(2n,q)$
orbit, we see the conjugacy classes of involutions are determined by $r$.   Moreover,
the centralizer of $g$ is obviously isomorphic to $Sp(2r,q) \times Sp(2n-2r,q)$ and so the size of the
conjugacy class of $g$ is
 \[ \frac{|Sp(2n,q)|}{|Sp(2r,q)| |Sp(2n-2r,q)|} .
\]
Since $r$ can take on any value between $0$ and $n$, this proves the theorem.
\end{proof}

The following theorem is our main result.

\begin{theorem} \label{Spoddlim} Let $i_{Sp}(2n,q)$ be the number of involutions in $Sp(2n,q)$.
\begin{enumerate}

\item Let $n$ be even and $q$ be odd and fixed. Then
\begin{eqnarray*}
\lim_{n \rightarrow \infty} \frac{i_{Sp}(2n,q)}{q^{n^2}} & = &
\frac{1}{2}
\left[ \prod_{i \geq 1} (1+1/q^{2i-1})^2 + \prod_{i \geq 1} (1-1/q^{2i-1})^2 \right] \\
& = &
\prod_{i\ge 1} \frac{(1+q^{4-8i})^2(1-q^{-8i})}{(1-q^{-2i})}.
\end{eqnarray*}

\item Let $n$ be odd and $q$ be odd and fixed. Then
\begin{eqnarray*}
\lim_{n \rightarrow \infty} \frac{i_{Sp}(2n,q)}{q^{n^2-1}} & = &
\frac{q}{2} \left[ \prod_{i \geq 1}
(1+1/q^{2i-1})^2 - \prod_{i \geq 1} (1-1/q^{2i-1})^2 \right] \\
& = & 2 \prod_{i\ge 1} \frac{(1+q^{-8i})(1-q^{-16i})}{(1-q^{-2i})}.
\end{eqnarray*}

\end{enumerate}
\end{theorem}

\begin{proof} Replacing $u$ by $uq$ in Theorem \ref{prodSp} gives
\begin{equation} \label{big3}
\sum_{n \geq 0} u^n q^{n^2+n} \sum_{r=0}^n
\frac{1}{|Sp(2r,q)| |Sp(2n-2r,q)|} = \frac{\prod_i (1+u/q^{2i-1})^2}{(1-u^2) \prod_i (1-u^2/q^{2i})}.
\end{equation}

Consider the $n \rightarrow \infty$ limit of the coefficient
of $u^n$ in the left hand side of \eqref{big3}. This is equal to
\begin{eqnarray*}
& & \lim_{n \rightarrow \infty}
\frac{1}{q^{n^2} (1-1/q^2) \cdots (1-1/q^{2n})} \sum_{r=0}^n
\frac{|Sp(2n,q)|}{|Sp(2r,q)| |Sp(2n-2r,q)|} \\
& = & \lim_{n \rightarrow \infty} \frac{1}{q^{n^2} (1-1/q^2) \cdots
(1-1/q^{2n})} i_{Sp}(2n,q) \\
& = & \frac{1}{\prod_i (1-1/q^{2i})} \lim_{n \rightarrow \infty}
\frac{i_{Sp}(2n,q)}{q^{n^2}}.
\end{eqnarray*} The first equality was Lemma \ref{countSp}.

Now consider the right hand side of \eqref{big3}. Except for poles at $u = \pm 1$, it is analytic in a disc of radius
greater than 1. Hence Lemma \ref{poles} gives that for $n$ even, the $n \rightarrow \infty$ limit of the coefficient of
$u^n$ in the right hand side of \eqref{big3} is
\[
\frac{1}{2} \left[ \frac{\prod_i (1+1/q^{2i-1})^2}{\prod_i (1-1/q^{2i})} +
\frac{\prod_i (1-1/q^{2i-1})^2}{\prod_i (1-1/q^{2i})} \right] .
\]
It follows that for $n$ even,
\[ \lim_{n \rightarrow \infty} \frac{i_{Sp}(2n,q)}{q^{n^2}} =
\frac{1}{2}
\left[ \prod_{i \geq 1} (1+1/q^{2i-1})^2 + \prod_{i \geq 1} (1-1/q^{2i-1})^2 \right].\] Letting $Q=1/q$ and using
Proposition \ref{sieveJTP} gives that this is equal to
\[ \frac{1}{2}\frac{1}{(Q^2;Q^2)_\infty}
\left( F(Q,Q^2) +F(-Q,Q^2)\right)\\
= \frac{1}{(Q^2;Q^2)_\infty} F(Q^4,Q^8).
\]

Similarly, if $n$ is odd, Lemma \ref{poles} gives that the $n \rightarrow \infty$ limit of the coefficient of $u^n$
in the right hand side of \eqref{big3} is
\[ \frac{1}{2} \left[ \frac{\prod_i (1+1/q^{2i-1})^2}{\prod_i (1-1/q^{2i})} -
\frac{\prod_i (1-1/q^{2i-1})^2}{\prod_i (1-1/q^{2i})} \right] .\] It follows that for $n$ odd,
\[ \lim_{n \rightarrow \infty} \frac{i_{Sp}(2n,q)}{q^{n^2}} = \frac{1}{2} \left[ \prod_{i \geq 1}
(1+1/q^{2i-1})^2 - \prod_{i \geq 1} (1-1/q^{2i-1})^2 \right]. \] Letting $Q=1/q$ and using Proposition
\ref{sieveJTP} gives that this is equal to
\[ \frac{1}{2}\frac{1}{(Q^2;Q^2)_\infty}
\left( F(Q,Q^2) - F(-Q,Q^2)\right) = \frac{1}{(Q^2;Q^2)_\infty} QF(Q^8,Q^8).\]

This proves the theorem. \end{proof}

{\it Remark:} The expression in part $1$ of Theorem \ref{Spoddlim} is equal to $1.1689..$ when $q=3$, and tends
to $1$ as $q \rightarrow \infty$. The expression in part $2$ of Theorem \ref{Spoddlim} is equal to $2.2819..$
when $q=3$, and tends to $2$ as $q \rightarrow \infty$.

Next we consider the case that $q$ is even.

To begin we use elementary means to derive a formula for the number of involutions in $Sp(2n,q)$ when $q$ is even.
For related results, see \cite{AS}, \cite{D1}, \cite{D2}, \cite{LS}.

\begin{theorem} \label{countSpeven} When $q$ is even, the number of involutions in $Sp(2n,q)$ is equal to
\[ \sum_{r=0 \atop r \ even}^n \frac{|Sp(2n,q)|}{A_r}
+\sum_{r=1 \atop r \ even}^n \frac{|Sp(2n,q)|}{B_r}+
\sum_{r=1 \atop r \ odd}^n \frac{|Sp(2n,q)|}{C_r} \]
where
$$
\begin{aligned}
A_r= &q^{r(r+1)/2+r(2n-2r)}|Sp(r,q)|\ |Sp(2n-2r,q)|,\\
B_r=& q^{r(r+1)/2+r(2n-2r)}q^{r-1}|Sp(r-2,q)|\ |Sp(2n-2r,q)|,\\
C_r=& q^{r(r+1)/2+r(2n-2r)}|Sp(r-1,q)|\ |Sp(2n-2r,q)|.
\end{aligned}
$$ \end{theorem}

The following lemma will be helpful for proving Theorem \ref{countSpeven}.

  \begin{lemma} \label{Spinv} Suppose $q$ is even, and let $g \in Sp(2n,q)$ be an involution.
  Let $V$ be the natural module of dimension $2n$ for $Sp(2n,q)$. Then
  \begin{enumerate}
  \item  $W=(g-1)V$ is totally singular of dimension $r$, where $r \le n$ is the rank of $g-1$.
  \item   $W^{\perp} = V^g$ is the fixed space of $g$ on $V$.
  \end{enumerate}
  \end{lemma}

  \begin{proof}   If $U$ is a nontrivial nondegenerate subspace of $V^g$, then
  clearly,  $V= U \oplus U^{\perp}$ is a $g$-stable decomposition whence
  the result follows by induction.

  So we may assume that $V^g$ is totally singular.  Since $\dim V^g \ge n$,
  it follows that $V^g$ is $n$-dimensional and is a maximal isotropic space.
  Since $g^2=1$,  $W \le V^g$ and since $\dim W + \dim V^g=2n$, we see
  in this case that $r=n$ and  $V^g=W^{\perp}$ and the result holds.
  \end{proof}

Now we prove Theorem \ref{countSpeven}.

\begin{proof} (Of Theorem \ref{countSpeven}) Let $G=Sp(2n,q)$, and let $g \in G$
be an involution. With respect to the flag $0 < (g-1)V <  V^g < V$,
and an appropriate basis, $g$ can be written as:

  $$
  \begin{pmatrix}     I   &  0  &  h  \\
                                  0 &  I    & 0    \\
                                  0 & 0 &   I
                                  \end{pmatrix}
                                  $$ where $h$ is a symmetric matrix of rank $r$.
  Let $P$ be the parabolic subgroup stabilizing the flag. Then $P$
  acts on the set of all such elements
  with any symmetric element in the upper corner by congruence
  (via $GL(r,q))$.

  We use the fact that two symmetric $m \times m$ matrices over
  a perfect field of characteristic $2$ are congruent if and only if they
  have the same rank and are both skew or are both nonskew (in particular
  if the rank is odd, then they are all congruent).

  Thus, we see that the conjugacy classes of involutions in $G$ are
  determined by the rank $r$ and if $r > 0$ is even on whether $h$ is skew
  or not. Clearly, two involutions written as above are conjugate in $G$
  if and only if they are conjugate in $P$ if and only if the right
  hand corner terms are congruent. In particular, there are $n+ 1 + n/2$ classes if $n$ is even
  and $n+1 + (n-1)/2$ classes if $n$ is odd.

  Now we can write down the centralizers of such elements.  Clearly
  any element of $C_G(g)$ preserves the flag and so $C_G(g)$ is contained
  in $P$.

  Note that $P=LQ$ where $Q$ is the unipotent radical with
  $$|Q|= q^{r(r+1)/2 + r(2n-2r)}.$$

  Clearly  $Q$ centralizes $g$.   Note that $L = GL(r,q) \times Sp(2n-2r, q)$
  and that $Sp(2n-2r,q)$ also centralizes $g$.  So we just need to compute
  the centralizer of $g$ in $K:=GL(r,q)$.   Now $A \in K$ acts on $g$ by sending
 $h$ to $AhA^{\top}$.   Thus,  $C_K(g) \cong Sp(r,q)$ if $r$ is even and
 $h$ is an alternating form.

 If  $h$ is not alternating, by the remarks above, we can take $h = I$ and so
 $C_K(g) = \{ A \in K | AA^{\top}=I\}$.  Of course, if $q$ were odd, this would
 just be an orthogonal group.  However, with $q$ even, this is not the case.
 We now compute the order of $C_K(g)$.

Let $U$ be an $r$-dimensional space with a nondegenerate symmetric
bilinear pairing $(,)$ on $U$  that is not alternating (so $(u,u)=1$
for some $u \in U$).   Let $U_1$ be the set of all $u$ with $(u,u)=0$.
Since we are in  a field of characteristic $2$, this is a hyperplane in $U$.
Note that the form restricted to $U_1$ is alternating.   Set $U_0 = U_1^{\perp}$.
Note that $U_0$ is $1$-dimensional (because our form is nondegenerate).

If $U_0 \cap U_1=0$, then $r$ is odd and $L$ preserves this decomposition whence
$C_K(g) \cong Sp(r-1,q)$  (the action is determined precisely by the action on $U_1$ since
it must be trivial on $U_0$ and preserve the splitting).

If $U_0 < U_1$, then $U_1/U_0$ is a nonsingular space of dimension
$r-2$ with the corresponding
alternating form and so $r$ is even.

In  this case any element of $C_K(g)$ is trivial on $U_0$.    Let $v \in U$
be a vector with $v$ outside of $U_0$ with $(v,v)=1$.   Then there exists
$x \in C_K(g)$ with $xv = v + u$ for any $u \in U_1$.   Moreover,  $C_K(g)$ induces
$Sp(r-2,q)$ on $U_0/U_1$ and it is an easy exercise to see that

$$
|C_K(g)| = |Sp(r-2,q)| q^{r-1}.
$$
 Summarizing, we can write down the orders of centralizers for each class
 representative:

 For each  even $r$ with $0 \le r \le n$, we have a class $C(r)$ with the centralizer having
 order
 $$A_r =
 q^{r(r+1)/2 + r(2n-2r)}|Sp(r,q)||Sp(2n-2r,q)|.
 $$

 For each $r$ with $1 \le r \le n$, we have a class $D(r)$ with centralizer having order
 for $r$ even:

 $$B_r=
 q^{r(r+1)/2 + r(2n-2r)} q^{r-1}|Sp(r-2,q)||Sp(2n-2r,q)|.
 $$

 and for $r$ odd:
 $$
 C_r=q^{r(r+1)/2 + r(2n-2r)} |Sp(r-1,q)||Sp(2n-2r,q)|.
 $$

Adding these contributions completes the proof.
\end{proof}

This leads to the following result.

\begin{theorem} \label{Spevenlim} Let $i_{Sp}(2n,q)$ be the number of involutions in $Sp(2n,q)$. Let $q$ be even and fixed. Then
\[ \lim_{n \rightarrow \infty} \frac{i_{Sp}(2n,q)}{q^{n^2+n}} = \prod_{i \geq 1} (1+1/q^{2i}) .\]
\end{theorem}

\begin{proof} Combining Theorem \ref{countSpeven} and Theorem \ref{Guralsum} gives that
\begin{equation} \label{spev}
 \sum_{n \geq 0} u^n q^{n^2} \frac{i_{Sp}(2n,q)}{|Sp(2n,q)|} = \frac{1}{1-u} \frac{\prod_{i \geq 1} (1+u/q^{2i})}
{\prod_{i \geq 1} (1-u^2/q^{2i})}. \end{equation}

Consider the $n \rightarrow \infty$ limit of the coefficient of $u^n$ in the left hand side of \eqref{spev}.
This is equal to
\begin{eqnarray*}
& & \lim_{n \rightarrow \infty} \frac{i_{Sp}(2n,q)}{q^{n^2+n} \prod_{i=1}^n (1-1/q^{2i})} \\
& = & \frac{1}{\prod_{i \geq 1} (1-1/q^{2i})} \lim_{n \rightarrow \infty} \frac{i_{Sp}(2n,q)}{q^{n^2+n}}.
\end{eqnarray*}

Now consider the right hand side of \eqref{spev}. Except for a pole at $u=1$, it is analytic in a disc
of radius greater than $1$. Hence Lemma \ref{poles} gives that the $n \rightarrow \infty$ of the
coefficient of $u^n$ in the right hand side of \eqref{spev} is
\[ \frac{\prod_{i \geq 1} (1+1/q^{2i})}{\prod_{i \geq 1} (1-1/q^{2i})}. \]

The theorem follows.
\end{proof}

{\it Remark:} The expression in Theorem \ref{Spevenlim} is equal to $1.3559..$ when $q=2$, and tends
to $1$ as $q \rightarrow \infty$.

\section{Orthogonal groups} \label{O}

This section studies the asymptotics of the number of involutions in orthogonal groups over a finite field
of size $q$. Recall
that the orthogonal group is the stabilizer of a nondegenerate quadratic form on a $d$-dimensional
space.
If $d=2n$ is even,
there are two types of nondegenerate quadratic forms and we denote their stabilizers by
$O^{\pm}(2n,q)$.   Note that $O^+(2n,q)$ is the stabilizer of a nondegenerate form that contains
totally singular subspaces of dimension $n$  (in the other case, the totally singular subspaces
have dimension at most $n-1$).  If $d$ is  odd, we only consider $q$ odd (as the orthogonal group
in characteristic $2$ in odd dimension is isomorphic to the symplectic group).

We assume
first that the characteristic is odd.

Lemma \ref{countO} gives an exact expression for the number of involutions (see also
\cite{V}). We note that the two terms in part 2 of Lemma \ref{countO} are equal.

\begin{lemma} \label{countO} Suppose that the characteristic is odd.
\begin{enumerate}
\item The number of involutions in $O^+(2n,q)$ is equal to
\[ \sum_{r=0}^{2n} \frac{|O^+(2n,q)|}{|O^+(r,q)||O^+(2n-r,q)|} + \sum_{r=1}^{2n-1} \frac{|O^+(2n,q)|}{|O^-(r,q)||O^-(2n-r,q)|} .\]

\item The number of involutions in $O^-(2n,q)$ is equal to
\[ \sum_{r=0}^{2n-1} \frac{|O^-(2n,q)|}{|O^+(r,q)||O^-(2n-r,q)|} + \sum_{r=1}^{2n} \frac{|O^-(2n,q)|}{|O^-(r,q)||O^+(2n-r,q)|} .\]

\item The number of involutions in $O^+(2n+1,q)$ (or in $O^-(2n+1,q)$) is equal to
\[ \sum_{r=0}^{2n+1} \frac{|O^+(2n+1,q)|}{|O^+(r,q)||O^+(2n+1-r,q)|} + \sum_{r=1}^{2n} \frac{|O^+(2n+1,q)|}{|O^-(r,q)||O^-(2n+1-r,q)|} .\]

\end{enumerate}
\end{lemma}

\begin{proof} Let $g \in O^{\pm}(n,q)$ be
an involution and let $V$ denote the natural orthogonal module.
Then $V = V_1 \perp V_{-1}$ where $V_1$ is the fixed space of
$g$ and $V_{-1}$ is the $-1$ eigenspace of $g$.  In particular, these
eigenspaces are nondegenerate.    Let $r = \dim V_1$.

The conjugacy class of $g$ is determined by the orbit of $V_1$ under
the orthogonal group.   Given a fixed $r$ with $0 < r < n$, there will be two orbits
depending upon the type of $V_1$  (note that the type of $V$ and $V_1$ determines
the type of $V_{-1}$).  If $r=0$ or $n$, then $g$ is a scalar.

Clearly the centralizer of $g$ is $O(V_1) \times O(V_{-1})$ and so
the size of the conjugacy class of $g$ is
$$
\frac{|O^{\pm}(n,q)|}{|O(V_1)||O(V_{-1})|},
$$
whence the result.

\end{proof}

This leads to the following result.

\begin{theorem} \label{mainO1}
Let $i_O^{\pm}(2n,q)$ denote the number of involutions in $O^{\pm}(2n,q)$. Let $q$
be odd and fixed. Then
\[ \lim_{n \rightarrow \infty} \frac{i_O^{\pm}(2n,q)}{q^{n^2}} = \prod_{i \geq 1} (1+1/q^{2i-1})^2 .
\]
\end{theorem}

\begin{proof} We prove the assertion for $i_O^+(2n,q)$ as the proof for $i_O^-(2n,q)$
is almost identical
(using part 2 of Lemma \ref{countO} instead of part 1 of Lemma \ref{countO},
and Theorem
\ref{prodO2} instead of Theorem \ref{prodO1}).

Replacing $u$ by $u/q$ in Theorem \ref{prodO1}, and applying part 1 of Lemma \ref{countO} gives
that \begin{equation}
\label{bigeq}
\frac{1}{2(1-u)} \frac{\prod_i (1+u/q^{2i-1})^2}{\prod_i (1-u^2/q^{2i})} + \frac{1}{2} \frac{\prod_i (1+u/q^{2i})^2}
{\prod_i (1-u^2/q^{2i})} \end{equation} is equal to
\begin{equation} \label{bigeq2}
\sum_{n \geq 0} \frac{u^n i_O^+(2n,q)}{2(1-1/q^n)(1-1/q^2) \cdots (1-1/q^{2(n-1)}) q^{n^2}}.
\end{equation}

Taking the $n \rightarrow \infty$ limit of the coefficient of $u^n$ in
\eqref{bigeq2} gives
\[
\frac{1}{2 \prod_i (1-1/q^{2i})} \lim_{n \rightarrow \infty}
\frac{i_O^+(2n,q)}{q^{n^2}} .
\]
Consider the $n \rightarrow \infty$ limit of the coefficient of $u^n$ in the
first infinite product in
\eqref{bigeq}; by Lemma \ref{poles}, it is equal to
\[ \frac{1}{2} \frac{\prod_i (1+1/q^{2i-1})^2}{\prod_i (1-1/q^{2i})} .
\]
Since the second term of
\eqref{bigeq} is analytic in a disc of radius bigger than $1$, it follows that
the $n \rightarrow \infty$ limit of the coefficient of $u^n$ in
the second infinite product of
\eqref{bigeq} is equal to $0$.

Summarizing, we have shown that
\[ \frac{1}{2 \prod_i (1-1/q^{2i})} \lim_{n \rightarrow \infty}
\frac{i_O^+(2n,q)}{q^{n^2}} =
 \frac{1}{2} \frac{\prod_i (1+1/q^{2i-1})^2}{\prod_i (1-1/q^{2i})},
\]
which implies the theorem.
\end{proof}

{\it Remark:} The expression in Theorem \ref{mainO1} is equal to $1.9296..$
when $q=3$, and tends to $1$ as $q \rightarrow \infty$.

We next treat the case of odd dimensional orthogonal groups in odd characteristic.

\begin{theorem} \label{mainO2} Let $i_O(2n+1,q)$ be the number of involutions in $O(2n+1,q)$. Let $q$
be odd and fixed. Then
\[ \lim_{n \rightarrow \infty} \frac{i_O(2n+1,q)}{q^{n^2+n}} = 2 \prod_{i \geq 1} (1+1/q^{2i})^2 .\]
\end{theorem}

\begin{proof} It follows from Theorem \ref{identlast} and part 3 of Lemma \ref{countO} that
\begin{equation} \label{last}
\sum_{n \geq 0} u^n q^{n^2} \frac{i_O(2n+1,q)}{|O(2n+1,q)|} =
\frac{1}{1-u} \frac{\prod_i (1+u/q^{2i})^2}{\prod_i (1-u^2/q^{2i})}.
\end{equation}

The $n \rightarrow \infty$ limit of the coefficient of $u^n$ on the left hand side of \eqref{last} is equal to
\[ \frac{1}{2 \prod_i (1-1/q^{2i})} \lim_{n \rightarrow \infty} \frac{i_O(2n+1,q)}{q^{n^2+n}}.\]
By Lemma \ref{poles}, the $n \rightarrow \infty$ limit
of the coefficient of $u^n$ on the right hand side of \eqref{last} is equal to
\[ \frac{\prod_i (1+1/q^{2i})^2}{\prod_i (1-1/q^{2i})} .\]

Comparing these two expressions proves the theorem.
\end{proof}

{\it Remark:} The expression in Theorem \ref{mainO2} is equal to $2.5382..$ when $q=3$, and tends to $2$
as $q \rightarrow \infty$.

Next we give results for even characteristic orthogonal groups.
Note that it is not necessary to treat odd dimensional
orthogonal groups in even characteristic, as these are isomorphic to symplectic groups.

To begin, we derive a formula for the number of involutions in even characteristic orthogonal groups.
Let $G=O^{\epsilon}(2n,q)$ with $q$ a power of $2$.    Let $\nu$ denote
 the quadratic form preserved by $G$ on the natural module $V$ of dimension $2n$.
 Recall that $\nu$ defines by  a nondegenerate $G$-invariant alternating form on $V$
 defined by
 $$
 (v,w): = \nu(v+w) + \nu(v) + \nu(w).
 $$

If $g \in G$ is an involution, then $g \in X = Sp(2n,q)$ with respect to the alternating
form defined above.

So by Lemma \ref{Spinv}, $W:=(g-1)V$ is totally singular with respect to the alternating form
of rank $r$.

We next note:

\begin{lemma}  If $x, g \in G$ are involutions and are conjugate in $X$, then
they are conjugate in $G$.
\end{lemma}

\begin{proof}
If $2n=2$ or $4$, this is clear.   So assume that $2n \ge 6$.

If $g$ is a transvection in $X$, then we see that $g$ is trivial on nondegenerate $2$-spaces
of either type (and so is $x$), whence the result follows by induction.

Suppose next that $g$ corresponds to a symmetric element in $X$ of rank at least $2$.
Then $g$ leaves invariant a nondegenerate $4$-space where it has $2$ Jordan blocks
of size $2$  with each nondegenerate.  Working in $Sp(4,q)$, we see that $g$ preserves
nondegenerate $2$-spaces of either type (where it acts nontrivially) and so the result
follows by induction.

Suppose that $g$ corresponds to a skew element in $X$.
We see from the section on even characteristic symplectic groups that $V$ is an orthogonal direct sum of $4$-dimensional
subspaces and one space where $g$ acts trivially.   If the last space is actually there,
we see by reducing to the case that $2n=6$, that $g$ acts trivially on $2$-dimensional
invariant subspaces of both types, whence the result follows by induction.  So the remaining
case is that all Jordan blocks have size $2$.   It follows that on each $4$-dimensional block
above, the type must be $+$ (such elements do not live in $O^{-}(4,q)$, which is isomorphic to
$SL(2,q^2).2$, the semidirect product of $SL(2,q^2)$ with a group of order 2 whose generator
induces the field automorphism on $SL(2,q^2)$ given by the $q$th power map on the field of
size $q^2$), whence the type of $G$ must be $+$ as well and the result is clear in this case.
\end{proof}

It also follows
from the argument above that each class in $X$ occurs in $G$ unless $g$ corresponds to
an alternating element in $X$ of rank $n$ (and so $n$ is even).  In the latter case,
$g$ can only be in $O^+(2n,q)$. In the proof of Theorem \ref{countSpeven}, the conjugacy
 classes of involutions in $Sp$ were called $C(r)$
and $D(r)$.  Let $C^{\epsilon}(r)$ and $D^{\epsilon}(r)$ denote the intersections with
$G$ (of type $\epsilon$).  By the previous lemma, this is either empty or a conjugacy
class in $G$.

Now we have to compute centralizers.

As noted if $g$ is in $C(r)$ with $r$ even, then $g$ is conjugate to an element of $G$ unless
$r =n$ and $\epsilon = -$.

Then  $g$ may be written as:
 $$
  \begin{pmatrix}     I   &  0  &  h  \\
                                  0 &  I    & 0    \\
                                  0 & 0 &   I
                                  \end{pmatrix},
                                  $$ where $h$ is skew of rank $r$.

We  see that  $g \in G_0: = SO^{\epsilon}(2n,q)$
and
$$
|C_{G}(g)|= |Sp(r,q)| |O^{\epsilon}(2n-2r,q)|q^{r(2n-2r)}q^{r(r-1)/2}.
$$

Note that if $r = n$, then the $O$ part of the centralizer is not there and $\epsilon = +$.
Note also that if $r < n$, the $G$ and $G_0$ classes are the same (because $g$ commutes
with a transvection in $G$)  while if $r=n$,
there is a single $G$-class that splits into $2$ $G_0$ classes (corresponding to the classes
of maximal totally isotropic subspaces).

We can compare  $|C^{\epsilon}(r)|$ and $|C(r)|$.   We have
$$
|C^{\epsilon}(r)|  =  [G:C_G(g)] =  \frac{|G| |C_X(g)|}{|X| |C_G(g)|}  |C(r)| = \frac{q^{n-r}  + \epsilon 1}{q^n + \epsilon 1}  |C(r)|.
$$

The above formula again shows that $|C^-(n)|=0$.

Next we consider the classes $D(r)$ with $1 \le r \le n$.   Note that for any such
element we can write $V$ as an orthogonal sum of nondegenerate
$2$-dimensional invariant spaces.   This implies that the $G$-classes and
$G_0$ classes coincide (since any such element commutes with a transvection
which is in $G$ but not in $G_0$).

Consider $X$ acting on the orthogonal module of dimension $2n+1$ (this is
an indecomposable $X$-module with a $1$-dimensional socle).   The stabilizers
of the complements to the $1$-dimensional socle are precisely the orthogonal
groups contained in $X$.  Of course, the number of complements is $q^{2n}$.
If $g \in D(r)$, then the number of $g$-invariant complements in $q^{2n-r}$.
Thus, $g$ lives in a total of $q^{2n-r}$ different orthogonal subgroups of $X$.

\begin{lemma}  Suppose that    $g \in D(r) \subset X$.
Then $g$  is in $q^{2n-r}/2$ orthogonal subgroups of $X$ of each type.
\end{lemma}

\begin{proof}  We can decompose $V = V_0 \perp V_1 \perp \ldots \perp V_r$
where each $V_i,  i > 0$ is a nondegenerate $2$-dimensional space on which
$g$ acts nontrivially and $g$ acts trivially on $V_0$.  Any quadratic form giving
rise to the given alternating form is determined by its restrictions to each $V_i$.
Clearly, there are the same number of quadratic forms of each type.
\end{proof}

Now we can compute $|D^{\epsilon}(r) |$.
Let $g \in D^{\epsilon}(r)$.

Let $\Omega$ be the set of orthogonal subgroups of $X$ of type $\epsilon$.
Since any two orthogonal subgroups of the same type are conjugate in $X$ and
their own normalizers (indeed, they are almost always maximal), we see
that
$$
|\Omega|= [X:O^{\epsilon}(2n,q)] =  q^n(q^n + \epsilon 1)/2.
$$

Let $\Omega(g)$ denote those elements of $\Omega$ containing $g$
(so this is the fixed point set of $g$ acting on $\Omega$).  By the previous lemma
$|\Omega(g)|=q^{2n-r}/2$.
An elementary formula (essentially coming from Burnside's Lemma) shows
that

$$
|\Omega(g)||D(r)| = |\Omega||D^{\epsilon}(r)|.
$$

Thus,
$$
|D^{\epsilon}(r)| =   \frac{ (1/2)q^{2n-r} } {|\Omega|}   |D(r)|  = \frac{ q^{2n-r} } {q^n(q^n + \epsilon 1)}   |D(r)| ,
$$
and so
$$
|D^{\epsilon}(r) | = \frac{q^{n-r} }{q^n + \epsilon 1} |D(r)|.
$$

Of course $|D(r)|$ was computed in the proof of Theorem \ref{countSpeven}.

Summarizing, we have proved the following theorem.

\begin{theorem} \label{countOeven} Let $q$ be even.
\begin{enumerate}
\item The number of involutions in $O^+(2n,q)$ is equal to

\[  \sum_{r=0 \atop r \ even}^n \frac{|O^+(2n,q)|}{A_r} + \sum_{r=1 \atop r \ even}^n \frac{|O^+(2n,q)|}{B_r} + \sum_{r=1 \atop r \ odd}^n \frac{|O^+(2n,q)|}{C_r}  \] where

\[ A_r = q^{r(r-1)/2+r(2n-2r)} |Sp(r,q)| |O^+(2n-2r,q)| \]

\[ B_r = 2q^{r(r+1)/2+(r-1)(2n-2r)-1} |Sp(r-2,q)| |Sp(2n-2r,q)| \]

\[ C_r = 2q^{r(r-1)/2+(r-1)(2n-2r)} |Sp(r-1,q)| |Sp(2n-2r,q)|. \]

\item The number of involutions in $O^-(2n,q)$ is equal to

\[ \sum_{r=0 \atop r \ even}^{n-1} \frac{|O^-(2n,q)|}{A_r} + \sum_{r=1 \atop r \ even}^n \frac{|O^-(2n,q)|}{B_r} + \sum_{r=1 \atop r \ odd}^n \frac{|O^-(2n,q)|}{C_r} \] where

\[ A_r = q^{r(r-1)/2+r(2n-2r)} |Sp(r,q)| |O^-(2n-2r,q)| \]

\[ B_r = 2q^{r(r+1)/2+(r-1)(2n-2r)-1} |Sp(r-2,q)| |Sp(2n-2r,q)| \]

\[ C_r = 2q^{r(r-1)/2+(r-1)(2n-2r)} |Sp(r-1,q)| |Sp(2n-2r,q)|. \]

\end{enumerate} Note that in the $A_r$ sum in part 2, $r$ only ranges from $0$ to $n-1$, and that the values of $B_r$
and $C_r$ are the same in parts 1 and 2 of the theorem.
\end{theorem}

This leads to the following result.

\begin{theorem} \label{Oevenmain} Let $i_O^{\pm}(2n,q)$ denote the number of involutions in $O^{\pm}(2n,q)$. Let $q$
be even and fixed. Then
\[ \lim_{n \rightarrow \infty} \frac{i_O^{\pm}(2n,q)}{q^{n^2}} = \prod_{i \geq 1} (1+1/q^{2i-1}). \]
\end{theorem}

\begin{proof} We prove the assertion for $i_O^+(2n,q)$ as the proof for $i_O^-(2n,q)$ is almost identical (using part
2 of Theorem \ref{countOeven} instead of part 1 of Theorem \ref{countOeven}, and Theorem \ref{GuralsumO-} instead of
Theorem \ref{GuralsumO+}).

Replacing $u$ by $u/q$ in Theorem \ref{GuralsumO+} and applying Theorem \ref{countOeven} gives that
\begin{equation} \label{first}
\frac{1}{2(1-u)} \frac{\prod_i (1+u/q^{2i-1})}{\prod_i (1-u^2/q^{2i})} + \frac{1}{2} \frac{\prod_i (1+u/q^{2i})}{\prod_i (1-u^2/q^{2i})}
\end{equation} is equal to
\begin{equation} \label{second}
\sum_{n \geq 0} \frac{u^n i_O^+(2n,q)}{2(1-1/q^n)(1-1/q^2) \cdots (1-1/q^{2(n-1)}) q^{n^2}}.
\end{equation}

Taking the $n \rightarrow \infty$ limit of the coefficient of $u^n$ in \eqref{second} gives
\[ \frac{1}{2 \prod_i (1-1/q^{2i})} \lim_{n \rightarrow \infty} \frac{i_O^+(2n,q)}{q^{n^2}}. \]
Consider the $n \rightarrow \infty$ limit of the coefficient of $u^n$ in the first infinite
product in \eqref{first}. By Lemma \ref{poles}, it is equal to \[ \frac{1}{2} \frac{\prod_i (1+1/q^{2i-1})}
{\prod_i (1-1/q^{2i})}.\] Since the second term of \eqref{first} is analytic in a disc of radius bigger than 1,
it follows that the $n \rightarrow \infty$ limit of the coefficient of $u^n$ in the second infinite
product of \eqref{first} is equal to $0$.

Summarizing, we have shown that
\[ \frac{1}{2 \prod_i (1-1/q^{2i})} \lim_{n \rightarrow \infty} \frac{i_O^+(2n,q)}{q^{n^2}}
= \frac{1}{2} \frac{\prod_i (1+1/q^{2i-1})}{\prod_i (1-1/q^{2i})}, \] which implies the theorem.
\end{proof}

{\it Remark:} The expression in Theorem \ref{Oevenmain} is equal to $1.7583..$ when $q=2$,
and tends to $1$ as $q \rightarrow \infty$.

\section{Acknowledgements} Fulman was partially supported by Simons Foundation Grant 400528.
Guralnick was partially supported by NSF grants DMS-1302886 and DMS-1600056. Stanton was partially
supported by NSF grant DMS-1148634. We are very grateful to Geoff Robinson for helpful
correspondence, and thank Ryan Vinroot for some references. We thank the referee for comments.


\begin{thebibliography}{A}

\bibitem{A} Andrews, G., The theory of partitions, Encyclopedia of Mathematics and its Applications, Vol. 2. Addison-Wesley
Publishing Co., Reading, Mass.-London-Amsterdam, 1976.

\bibitem{AS} Aschbacher, M. and Seitz, G., Involutions in Chevalley groups over fields of even order, {\it Nagoya Math. J.} {\bf 63}
(1976), 1-91.

\bibitem{D1} Dye, R. H., On the conjugacy classes of involutions of the orthogonal groups over perfect fields of characteristic 2, {\it
Bull. London Math. Soc.} {\bf 3} (1971), 61-66.

\bibitem{D2} Dye, R. H., On the conjugacy classes of involutions of the simple orthogonal groups over perfect fields of characteristic
two, {\it J. Algebra} {\bf 18} (1971), 414-425.

\bibitem{FV} Fulman, J. and Vinroot, R., Generating functions for real character degree sums of finite general linear and unitary groups,
{\it J. Algebr. Comb.} {\bf 40} (2014), 387-416.

\bibitem{GG} Garibaldi, S. and Guralnick, R., Spinors and essential dimension,  {\it Compos. Math}, to appear,  arXiv:1601.00590 (2016).

\bibitem{GP} Ginzburg, V. and Pasechnik. D., Random chain complexes, arXiv:1602.08538 (2016).

\bibitem{LS} Liebeck, M. and Seitz, G., Unipotent and nilpotent classes in simple algebraic groups and Lie algebras, Mathematical
Surveys and Monographs, 180. American Math Society, Providence, RI, 2012.

\bibitem{M} Morrison, K., Integer sequences and matrices over finite fields, {\it J. Integer Seq.} {\bf 9} (2006), Article 06.2.1, 28 pp.

\bibitem{O} Odlyzko, A., Asymptotic enumeration methods, Chapter 22 in Handbook of Combinatorics, Volume 2. MIT Press and Elsevier, 1995.

\bibitem{V} Vinroot, R., Character degree sums and real representations of finite classical groups of odd characteristic, {\it J.
Algebra Appl.} {\bf 9} (2010), 633-658.

\end{thebibliography}
\end{document}